\documentclass[a4paper,oneside]{amsart}

\def\Wdn#1\wdn{\marginpar{\tiny #1}}
\long\def\WDN#1\wdn{[WDN: #1]\Wdn[Comment]\wdn}
\usepackage{amssymb}
\usepackage{amsmath}
\usepackage{amsfonts}
\usepackage{amsthm}
\usepackage{mathrsfs}
\usepackage{fix-cm}
\usepackage{graphicx}
\usepackage{amscd}

\usepackage{soul}
\sodef\spred{}{.2em}{.9em plus.4em}{1em plus.1em minus.1em}

\usepackage[latin1]{inputenc}
\usepackage[T1]{fontenc}

\usepackage[english]{babel}

\usepackage{bbm}

\usepackage[all]{xy}
\newbox\mybox
\def\overtag#1#2#3{\setbox\mybox\hbox{$#1$}\hbox to
  0pt{\vbox to 0pt{\vglue-#3\vglue-\ht\mybox\hbox to \wd\mybox
      {\hss$\ss#2$\hss}\vss}\hss}\box\mybox}
\def\undertag#1#2#3{\setbox\mybox\hbox{$#1$}\hbox to 0pt{\vbox to
    0pt{\vglue#3\vglue\ht\mybox\hbox to \wd\mybox
      {\hss$\ss#2$\hss}\vss}\hss}\box\mybox}
\def\lefttag#1#2#3{\hbox to 0pt{\vbox to 0pt{\vss\hbox to
      0pt{\hss$\ss#2$\hskip#3}\vss}}#1}
\def\righttag#1#2#3{\hbox to 0pt{\vbox to 0pt{\vss\hbox to
      0pt{\hskip#3$\ss#2$\hss}\vss}}#1}
\let\ss\scriptstyle

\def\Dot{\lower.2pc\hbox to 2.5pt{\hss$\bullet$\hss}}
\def\Circ{\lower.2pc\hbox to 2.5pt{\hss$\circ$\hss}}
\def\Vdots{\raise5pt\hbox{$\vdots$}}
\def\splicediag#1#2{\xymatrix@R=#1pt@C=#2pt@M=0pt@W=0pt@H=0pt}

\newcommand\lineto{\ar@{-}}
\newcommand\dashto{\ar@{--}}
\newcommand\dotto{\ar@{.}}

\newtheorem{thm}{Theorem}[section]

\newtheorem{prop}[thm]{Proposition}
\newtheorem{cor}[thm]{Corollary}

\newtheorem{thm*}{Theorem}

\theoremstyle{definition} 
\newtheorem{defn}[thm]{Definition} 

\newtheorem{rmk}[thm]{Remark}

\newcommand{\Q}{\mathbbm{Q}}

\newcommand{\Z}{\mathbbm{Z}}

\newcommand{\num}[1]{\lvert #1 \rvert}

\newcommand{\lcm}{\operatorname{lcm}}

\DeclareMathOperator*{\conect}{\#}

\newcommand{\morf}[4][\to]{ #2 \colon #3 #1 #4}
\newcommand{\inv}{^{-1}}

\makeatletter
\newcommand\@b@gconect[1]{%
\vcenter{\hbox{#1$\m@th\mkern2mu\conect\mkern2mu$}}}
\newcommand\@bigconect{%
\mathchoice{\@b@gconect\huge} % display style
{\@b@gconect\LARGE} % text style
{\@b@gconect{}} % script style
{\@b@gconect\footnotesize} % script script style
}
\newcommand\bigconect{\mathop{\@bigconect}\displaylimits}
\makeatother

\renewcommand{\phi}{\varphi}
\renewcommand{\epsilon}{\varepsilon}

\begin{document}

\bibliographystyle{alpha}

%\tableofcontents

\title[When is the universal abelian cover a $\Q$HS]
{Determining when the universal abelian cover of a graph manifold is a
  rational homology sphere.}
\author{Helge M\o{}ller Pedersen}
\address{Matematische Institut\\ Universität Heidelberg
\\ Heidelberg, 69120}
\email{pedersen@mathi.uni-heidelberg.de}
\begin{abstract}
  It was shown in \cite{myarticle} that the splice diagram of a
  rational homology sphere graph manifold determines the manifolds
  universal abelian cover. In this article we use the proof given in
  \cite{myarticle} to give a condition on the splice diagram to
  determine when the universal abelian cover itself is a rational
  homology sphere.
\end{abstract}
\maketitle

\section{Introduction}

Graph manifolds is an interesting class of $3$-manifolds. They are
defined as the manifolds who only have Seifert fibered pieces in their
JSJ-decomposition, or equivalently have no hyperbolic pieces in their
geometric decomposition. They are also the $3$-manifolds that are
boundary of plumbed $4$-manifolds, and therefore all links of
isolated complex surface singularities are graph manifolds. 

If we restrict to rational homology sphere graph manifolds, then there are
interesting question involves the universal abelian cover. The first is
of course when do two manifolds have the same universal abelian cover. A answer
to this was given in \cite{myarticle} using an
invariant called splice diagram, saying that if two graph manifolds
have the same splice diagram, then their universal abelian covers are
homeomorphic. There I gave a simple corollary:

%\section{Determining when the universal abelian cover is a rational
%  homology sphere}

%In this chapter we will give some corollaries of the proof of the
%second main theorem. We determine from the splice
%diagram when the universal abelian cover is a rational or an integer
%homology sphere. We begin by recognizing manifolds with
%integer homology sphere universal abelian covers.

\begin{cor}
Let $M$ be a rational homology sphere graph manifold with splice
diagram $\Gamma(M)$, such that around any node in $\Gamma(M)$ the
edge weights are pairwise coprime. Then the universal abelian cover of
$M$ is an integer homology sphere.
\end{cor}

The present article will strengthen this result and answer when is the
universal abelian cover a rational homology sphere. We are going to do
this, by investigating the construction of the universal abelian cover
from the splice diagram given in \cite{myarticle} see also
\cite{constructinabcovers} for this construction in more algorithmic
form. 

The splice diagram we use differs slightly from the original
definition given in \cite{Siebenmann} and \cite{EisenbudNeumann} by
only having non negative weights at edges and not demanding that the
edges at a node are pairwise coprime, the last is of course because we
are working with rational homology spheres and not only
integer homology spheres. Our splice diagram also differ slightly
from the once in \cite{NeumannWahl3}, \cite{neumannandwahl1}, and
\cite{neumannandwahl2}, by having signs at nodes, but for singularity
links which is what concern Neumann and Wahl in those articles, our
splice diagram are the same. 

This article has two sections. In the first we introduce splice
diagrams and and give some result about them need in the second
section where we prove a condition for when the universal abelian
cover is a rational homology sphere. This result was originally partly
in my Ph.d.\ thesis, but I was at that time not able to prove what
here is Proposition \ref{intersectionform}, and could therefore only
show sufficiency of the condition in the case of singularity links,
using that finite branched covers only branched over the singular
point of singularity links are themselves singularity links. Because 
this was originally part of my thesis I would like to thank Ben Elias
who helped me editing my thesis and my Ph.d.\ advisor Walter Neumann.

%\begin{proof}
%It is shown in \cite{EisenbudNeumann} that a splice diagram with
%pairwise coprime edge weights at nodes is the
%splice diagram for an integer homology sphere.
% So given a splice
%diagram satisfying the assumption there is a integer homology sphere
%$M'$ with splice diagram $\Gamma(M)$, hence $M'$ is the universal
%abelian cover of $M$ by Theorem \ref{universalabcover}. 
%\end{proof}   

%This actually gives a way to construct the universal abelian cover for any
%rational homology sphere graph manifold  that
%has a splice diagram with 
%pairwise coprime edge weights at nodes, since \cite{EisenbudNeumann}
%describes how to construct the integral homology sphere by splicing,
%and to construct a plumbing diagram for it in the case where we have
%positive decorations at nodes. If we want a construction of the
%plumbing diagram when there are
%negative decorations at nodes, Theorem 3.1 in
%\cite{bilinearforms} contains a method whereby given a splice diagram
%$\Gamma(M)$ as above, one can construct an unimodular tree
%$\Delta(M)$ which will be the plumbing diagram.

\section{Splice diagrams}

A \emph{splice diagram} is a weighted tree with no vertices of valence
two, with signs on \emph{nodes}, that is vertices of valence tree or
higher, and with non negative integers on edges adjacent to nodes.
$$\splicediag{8}{30}{
  \Circ&&&&\Circ&\\
  &\oplus  \lineto[ul]_(.25){3}
  \lineto[dl]^(.25){5}
  \lineto[rr]^(.25){22}^(.75){10}&& \ominus
  \lineto[ur]^(.25){7}
  \lineto[dr]_(.25){2}^(.75){6}&&\Circ\\
  \Circ&&&&\oplus\lineto[ur]^(.25){3}\lineto[dr]_(.25){2}&\\
&&&&&\Circ
\hbox to 0 pt{~,\hss} }$$

We want to assign a splice diagram $\Gamma(M)$ to any given $\Q$HS
graph manifold $M$, this is done in the following way:
\begin{itemize}
\item Take a vertex for each of the Seifert fibered pieces of the
  JSJ-decomposition of $M$, these are the vertices there are going to
  be the nodes of $\Gamma(M)$ and we will hence forward not
  distinguish between a node and the Seifert fibered piece it
  represents.
\item Connect two nodes by an edge if they are they are glued to make
  $M$ from the pieces.
\item Add a vertex and connect it with an edge to a node for for each
  singular fiber of the node, we will call these vertices for leaf, and
  will in general not distinguish between a leaf and the edge leading
  to it.
\item The sign added at a node is the linking number of two non
  singular fibers, see \cite{myarticle} for more details.
\item If $e$ is an edge at a node $d$, the it corresponds to a torus
  $T_e\subset M$, either from the JSJ-decomposition or as the boundary
  of a tubular neighborhood of a singular fiber. Let $M_{ve}'$ be the
  connected piece of $M-T_e$ not containing $v$, and let
  $M_{ve}=M_{ve}'\bigcup(D^2\times S^1)$, where we identify a meridian
  of the solid torus with the image of a fiber of $v$. Then the edge
  weight $d_{ve}$ is $\num{H_1(M_{ve})}$, if $H_1(M_{ve})$ is infinite
  $d_{ve}=0$.  
\end{itemize}

The assumption that $M$ is a $\Q$HS insure that this construction
gives a tree, since the decomposition structure (or decomposition
graph see \cite{commensurability}) of a $\Q$HS is a tree, a fact which
will be used later when we find obstructions to the universal abelian
cover being a $\Q$HS. If we consider $M$ as a plumbed manifold, then
we can give the following characterization of being a $\Q$HS.

\begin{prop}\label{rationalhomologyspheres}
Let $M$ be a plumbed $3$-manifold, then $M$ is a $\Q$HS is and only
if its plumbing diagram is a tree of spheres, and its intersection
form is non degenerate.
\end{prop}

One can construct the splice diagram of $M$ from a plumbing diagram of
$M$, see \cite{myarticle} for details. This is in fact the way splice
diagrams are defined in \cite{NeumannWahl3},
\cite{neumannandwahl2}, and \cite{neumannandwahl1}.

From the splice diagram there are several numerical invariants that
plays important roles. The first is the \emph{edge determinant} which
is a number associated to an edge $e$ between nodes $v_0$ and $v_1$, if
the nodes look like
$$\splicediag{8}{30}{
  &&&&\\
  \Vdots&\overtag\Circ {v_0} {8pt}\lineto[ul]_(.5){n_{01}}
  \lineto[dl]^(.5){n_{0k_0}}
  \lineto[rr]^(.25){r_0}^(.75){r_1}&& \overtag\Circ{v_1}{8pt}
  \lineto[ur]^(.5){n_{11}}
  \lineto[dr]_(.5){n_{1k_1}}&\Vdots\\
  &&&&\hbox to 0 pt{~.\hss} }$$ 
then it is defined as 
\begin{align*}
r_0r_1- \epsilon_0\epsilon_1
\big(\prod_{i=1}^{k_0}n_{oi}\big)\big(\prod_{j=1}^{k_1}n_{1j}\big),
\end{align*}
where $\epsilon_i$ is the sign on the $i$'th node.

The edge determinant is important in the following two result from
\cite{myarticle}

\begin{prop}[Edge Determinant Equation]\label{edgedeterminantequation}
Let $e$ be an edge between two nodes $v$ and $w$, let $D(e)$ be its edge
determinant. then the fiber intersection number $p$ in the
corresponding torus is given by
\begin{align*}
p&=\frac{\num{D}}{\num{H_1(M)}}.
\end{align*}
\end{prop}
By the fiber intersection number one means, the intersection number in
the torus of a fiber from each of the two Seifert fibered pieces $v$
and $w$ with appropriate chosen orientations, for more details see
\cite{commensurability} or \cite{myarticle}.

\begin{thm}
A $\Q$HS graph manifold is the link or an isolated complex surfaces
singularity if and only if the are no negative signs in its splice
diagram and all edge determinants are positive.
\end{thm} 

The other numbers derived from the splice diagram we need is called
\emph{ideal generators} and is associated to a vertex $v$ and adjacent
edge $e$. To define it we first need following construction. Let $v$
and $w$ be two vertices in $\Gamma$, then their \emph{linking number}
$l_{vw}$ is the product of all edge weights adjacent two but not on
the shortest path from $v$ to $w$. We define $l_{vw}'$ similar except
we do not include the weights adjacent to $v$ and $w$. Then if $v$ is
a node and $e$ and adjacent edge we define
an ideal of $\Z$ by
\begin{align*}
I_{ve}&=\langle l_{vw}'\vert\ w\text{ is a leaf of }\Gamma_{ve}\rangle
\end{align*}
 where $\Gamma_{ve}$ is the connected component of $\Gamma-e$ not
 including $v$. We the define the ideal generator $\overline{d}_{ve}$
 as the positive generator of $I_{ve}$.

The ideal generator is important because of the following \emph{ideal
  condition}. 

\begin{defn}
A splice diagram is said to satisfy the ideal condition if for any
node $v$ and adjacent edge $e$, then the ideal generator
$\overline{d}_{ve}$ divides the edge weight $d_{ve}$. 
\end{defn}

Every splice diagram coming from a graph manifold satisfy the ideal
condition. This follows from 
the following topological description of the ideal generator from
\cite{neumannandwahl2}.

\begin{thm}\label{idealcon}
Let $M$ be a $\Q$HS and $\Gamma(M)$ its splice diagram, let $v$ be a
node and $e$ an adjacent edge, then
$\overline{d}_{ve}=\num{H_1(M_{ve},K)}$. Where $K$ is the core of the
solid torus one glue to $M_{ve}'$ to construct $M_{ve}$.
\end{thm}

Since our proofs in the next section relies on the combinatorics of
splice diagram, we will introduce some helpful notation. 
\begin{defn}\label{sees} 
We say that an edge weight $r$ of a splice diagram \emph{sees} a vertex $v$ (or
edge $e$) of the splice diagram if, when we delete the node which $r$
is adjacent to, the vertex $v$ (or the edge $e$) and the edge which
$r$ is on are in the same connected component. 
\end{defn}

\begin{rmk}\label{notationaboutsee}
Given a vertex $v$ on any edge $e$ between nodes, one of the
edge weights at $e$ sees $v$ and the other does not see $v$.
Let us introduce the following notation. Let $v$ be a vertex
of $\Gamma$ and let $v'$ be a node of $\Gamma$, where $v\neq v'$. Then
let  $r_{v'}(v)$ be the unique edge weight at an edge adjacent to $v'$
which sees $v$. Likewise let $d_{v'}(v)$ be the unique ideal generator
associated to $v'$, which sees $v$.
\end{rmk}

\begin{defn}
We say that an edge weight $r_v$ sees an edge weight $r_{v'}$ if $r_v$
sees $v'$ and $r_{v'}\neq r_{v'}(v)$. Likewise for ideal generators. 
\end{defn}

\begin{prop}\label{idealgeneratorproperties}
Let $v$ be a node of a splice diagram $\Gamma$ of a manifold $M$. Let
$r_v$ be a edge weight adjacent to $v$ and let $d_v$ be the
corresponding ideal generator. Then $r_v$ and $d_v$ are divisible by
every ideal generator they see. Moreover if $r_v$ and $d_v$ see a
node $v'$ and $n,n'$ are edge weights at $v'$ and $n,n'\neq r_{v'}(v)$,
then $\gcd(n,n')\mid r_v,d_v$.
\end{prop}

\begin{proof}
We first observe that it is enough to show the proposition only for
$d_v$, since $d_v\mid r_v$ by \ref{idealcon}. Let  $e$ be the edge
$d_v$ is on. 

 We will show
this by induction on the number of edges between $v$ and $v'$. If
$e$ is adjacent to $v'$, then $d_{v}$ is the generator of an ideal $I$,
which can be generated by elements, each of which is divisible by the
product of all but one of the edge weights at $v'$ not on
$e$. But this implies that each of the elements in the
generating set is divisible by either $n$ or $n'$, and hence each the
elements is divisible by $\gcd(n,n')$, and therefore $d_{v}$ is divisible by
$\gcd(n,n')$. 

Assume by induction that if there are $k$ edges between $v''$ and $v'$
then $d_{v''}(v')$ are divisible by
$\gcd(n,n')$. Assume that there
  are $k+1$ edges between $v$ and $v'$. 
Let $d_i$ for $i\in 1,\dots,k$ be $d_{\widetilde{v}}(v')$ on the
vertex $\widetilde{v}$ on $i$'th edge between $v$ and $v'$, then by induction
$\gcd(n,n')\mid d_i$ for all $i$. Remember that $d_{v}$ is the
generator of the ideal 
\begin{align}
\langle l_{vw}'\vert w \text{ is a leaf of }\Gamma_{ve}\rangle. 
\end{align}    
where $l_{vw}'$ is the product of the edge weights adjacent to but not on
the path from $v$ to $w$. Now there are two types of leaves $w$: the
leaves $w$ where the path between $w$ and $v$ goes through $v'$, and
the one where the path does not go through $v'$. In the first case,
$n\mid l_{v'w}'$ or $n'\mid l_{vw}'$ or both, so in this case
 $\gcd(n,n')\mid l_{vw}'$. In the second case one of the $d_i$'s will
 divide $l_{vw}'$. This
 implies that $\gcd(n,n')\mid l_{vw}'$ for all $w$, and hence
 $\gcd(n,n')$ divides the generator of the ideal $d_{v}$. 
The following illustrates how $\Gamma$ looks in the first and the
second case of the induction. In the first case, the path to $w$ can
also pass through the edges with $n$ or $n'$.
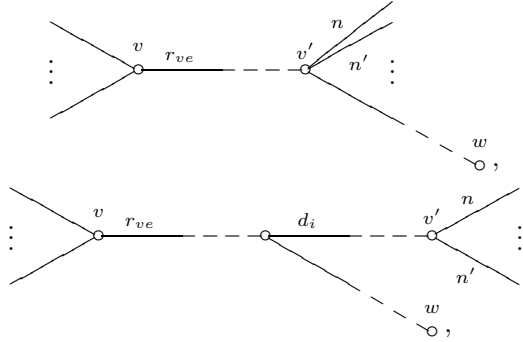
$$\splicediag{8}{30}{
&&&&&\\
&&&&&\\
  \Vdots&\overtag\Circ {v} {8pt}\lineto[ul]
  \lineto[dl]
  \lineto[r]^(.5){r_{ve}}&\dashto[r]& \overtag\Circ{v'}{8pt}
  \lineto[uur]^(.5){n}\lineto[ur]_(.5){n'}
  \lineto[dr]&\Vdots&\\
&&&&\dashto[ddr]&\\ 
&&&&&\\ 
&&&&&\overtag\Circ {w} {8pt}
\hbox to 0 pt{~,\hss} }$$
$$\splicediag{8}{30}{
&&&&&&\\
  \Vdots&\overtag\Circ {v} {8pt}\lineto[ul]
  \lineto[dl]
  \lineto[r]^(.5){r_{ve}}&\dashto[r]&
  \Circ\lineto[r]^(.5){d_i}\lineto[dr]&\dashto[r]& 
 \overtag\Circ{v'}{8pt}
  \lineto[ur]^(.5){n}\lineto[dr]_(.5){n'}&\Vdots\\
&&&&\dashto[ddr]&&\\  
&&&&&&\\
&&&&&\overtag\Circ {w} {8pt}
\hbox to 0 pt{~,\hss} }$$

The statement about $d_v$ being divisible by ideal generators it sees
follows from a similar argument as above.
\end{proof}

\section{Main Theorems}

To determine conditions on the splice diagram
for the universal abelian cover to be a rational homology sphere, we
investigate the construction of the universal abelian
cover in Theorem 6.3 of \cite{myarticle}. Since we construct the
universal abelian cover by induction, there are two places where
obstructions to being a rational homology sphere can arise: in the
inductive step, and in the base case. 

We start by looking at the base case, that is a splice diagram with
one node. We distinguish between diagrams with an edge weight of 0 and
those without. In the
case of an edge weight of $0$, we never get rational homology sphere
universal abelian covers. The universal abelian cover $X$ of
$L(p,q)\conect L(p',q')$ is $p$ copies of $S^3$ with $p'$
balls removed, glued to $p'$ copies of $S^3$ with $p$ balls removed,
where the former pieces are glued to the latter pieces exactly once
each. Then a Meyer-Vietoris argument shows that the rank
of the first homology group is $(p-1)(p'-1)$. Since the universal
abelian covers of iterated connected sums of lens spaces will contain several
copies of $X$ as connected summands, it is clear that a connected sum of lens
space can not have rational homology sphere universal abelian covers.

This leaves the second case, determining which Seifert fibered,
or more precisely, which $S^1$ orbifold bundles have rational homology
sphere universal abelian covers. By the results of \cite{neumann832} and
\cite{geometryss}, which also works for graph orbifolds, this is the same as
determining which links of Brieskorn complete intersections are
rational homology spheres.  

\begin{prop}\label{rationalhomologyspherebrieskorn}
$\Sigma(\alpha_1,\dots,\alpha_n)$ is a rational homology sphere if and
only if one of the three following conditions holds.
\begin{enumerate}
\item $\gcd(\alpha_i,\alpha_j)=1$ for all $i\neq j$.
\item There exist
  a single pair $k,l$, such that $\gcd(\alpha_k,\alpha_l)\neq 1$.
\item There exist a single triple $k,l,m$ such that
  $\gcd(\alpha_l,\alpha_k)=\gcd(\alpha_l,\alpha_m)=\gcd(\alpha_m,\alpha_k)=2$;
  for all other indices $\gcd(\alpha_i,\alpha_j)=1$.
\end{enumerate}
\end{prop} 

The first condition is of course the case where
$\Sigma(\alpha_1,\dots,\alpha_n)$ is a integer homology sphere, as we
saw earlier.

\begin{proof}
The if direction follows from \cite{hamm}, where Hamm proves a
sufficient condition for the link of Brieskorn complete intersections of any
dimension to be rational homology spheres. He could only prove the
other direction if the number of variables was at most twice the
dimension plus two. We will give a different proof in the case of
surfaces, using the description of the Seifert invariants given in
Theorem 2.1 in \cite{neumannandraymond}. 

A Seifert fibered manifold is a
rational homology sphere if and only if the rational euler number
$e$ is nonzero, and the genus $g$ is zero. From the formulas of Theorem
2.1 in \cite{neumannandraymond} we see that
$e(\Sigma(\alpha_1,\dots,\alpha_n))\neq 0$, so we need only show that
the conditions above are equivalent to the genus being $0$. In other
words it is enough to show that the following equation holds if and
only if one of the three conditions does:
\begin{align}\label{genuseqzero}
0=2+(n-2)\frac{\prod_i\alpha_i}{\lcm_i(\alpha_i)}-
\sum_{i=1}^n\frac{\prod_{j\neq i}\alpha_j}{\lcm_{j\neq i}(\alpha_j)}.
\end{align}
Let $A=\frac{\prod_i\alpha_i}{\lcm_i(\alpha_i)}$ and
$A_i=\frac{\prod_{j\neq i}\alpha_j}{\lcm_{j\neq i}(\alpha_j)}$. 
%Notice we have that $A=\prod_{i<j}\gcd(\alpha_i,\alpha_j)$,
%$A_i=\prod_{\substack{j<k  \\ j,k\neq i}}\gcd(\alpha_j,\alpha_k)$ and
%$\tfrac{A}{A_i}=\prod_{j\neq i}\gcd(\alpha_i,\alpha_j)$. 

We start by proving the ``if'' direction. Assume condition $1$ holds, then
$A=1$ and $A_i=1$ for all $i\in 1,2,\dots,n$, and we get
\begin{align}
g=2+(n-2)A-\sum_{i=1}^nA_j=2+(n-2)+\sum_{i=1}^n1=2+(n-2)-n=0
\end{align}
Assume that condition $2$ holds, and let $\gcd(\alpha_k,\alpha_l)=B$. Then
$A=B$, $A_k=A_l=1$ and $A_i=B$ if $i\neq k,l$. We get
\begin{align}
g=2+(n-2)B-1-1-\sum_{i\neq
  k,l}B=2+(n-2)B-2-(n-2)B=0
\end{align}
Finally for condition $3$, $A=4$, $A_k=A_l=A_m=2$ and
$A_j=4$ if $j\neq k,l,m$. The genus is
\begin{align}
g=2+(n-2)4-2-2-2-\sum_{i\neq
  k,l,m}4=(n-2)4-4-(n-3)4=0.
\end{align}
This conclude the ``if'' direction. 

For the ``only if'' direction we start by assuming the equation
\eqref{genuseqzero} 
holds. Suppose we have $\alpha_j,\alpha_k,\alpha_l,\alpha_m$,
such that $\gcd(\alpha_j,\alpha_k)=B$ and
$\gcd(\alpha_l,\alpha_m)=C$. Notice that $BC\mid A$, $B\mid A_l,A_m$,
$C\mid A_j,A_k$ and $BC\mid A_i$ for $i\neq j,k,l,m$. Let
$A'=\frac{A}{BC}$, $A_j'=\frac{A_j}{C}$, $A_k'=\frac{A_k}{C}$,
$A_l'=\frac{A_l}{B}$, $A_m'=\frac{A_m}{B}$ and $A_i'=\frac{A_i}{BC}$
if $i\neq j,k,l,m$. $A\geq A_i$ for all $i$ so clearly $A'\geq A_i'$ for
$i\neq j,k,l,m$. 
If $i=j,k$ then $B\mid\frac{A}{A_i}$, and if
$i=l,m$ then $C\mid\frac{A}{A_i}$, so we also
get $A'\geq A_i'$. 

Hence 
\begin{align}
 0&=2+(n-2)BCA'-CA_j'-CA_k'-BA_l'-BA_m'-\sum_{s\neq j,k,l,m}BCA_s'\nonumber\\
   &\geq 2+(n-2)BCA'-2CA'-2BA'-\sum_{s\neq j,k,l,m}BCA'=2+2A'(BC-C-B).
\end{align}
Since $A'\geq1$ this implies that $BC-C-B<0$ and hence either $B=1$ or
$C=1$. 

We have now proved that $\gcd(\alpha_i,\alpha_j)=1$ except that there
might be $\alpha_k,\alpha_l,\alpha_m$ such that
$\gcd(\alpha_k,\alpha_l)=B$ and $\gcd(\alpha_k,\alpha_m)=C$ and
$\gcd(\alpha_l,\alpha_m)=D$. Notice that
$A=\frac{\alpha_k\alpha_l\alpha_m}{\lcm(\alpha_k,\alpha_l,\alpha_m)}$,
$A_k=\frac{\alpha_l\alpha_m}{\lcm(\alpha_l,\alpha_m)}$,
$A_l=\frac{\alpha_k\alpha_m}{\lcm(\alpha_k,\alpha_m)}$, 
$A_m=\frac{\alpha_k\alpha_l}{\lcm(\alpha_k,\alpha_l)}$, 
and $A_i=A$ for $i\neq k,l,m$. The equation becomes 
\begin{align}
 0&=2+(n-2)A-A_k-A_l-A_m-\sum_{s\neq k,l,m}A_j\nonumber\\
   &=2+(n-2)A-A_k-A_l-A_m-\sum_{s\neq k,l,m}A=2+A-A_k-A_l-A_m,
\end{align}
which is exactly the same equation as if $n=3$. But it is known in
this case that either $B=C=D=2$ or two of $B,C,D$ is $1$, from
the article of Hamm \cite{hamm}. One can also see this directly, if
$\alpha_1=ds_2s_3t_1$, $\alpha_1=ds_1s_3t_2$, and
$\alpha_3=ds_1s_2t_3$ where $\gcd(s_i,s_j)=1$ and
$\gcd(t_i,t_j)=1$, then
the equation becomes $0=2+d^2s_1s_2s_3-d(s_1+s_2+s_3)$. It is clear
that $d=1$ or $2$. If $d=2$ then the only solution is $s_1=s_2=s_3=1$
since the right hand side is increasing in $s_i$. If $d=1$ then the
only solution  is if two of the $s_i$'s are one,
since the right hand side increases if we increase two of the $s_i$'s. 

%We have now proved that $\gcd(\alpha_i,\alpha_j)=1$ except that there
%might be $\alpha_k,\alpha_l,\alpha_m$ such that
%$\gcd(\alpha_k,\alpha_l)=B$ and $\gcd(\alpha_k,\alpha_m)=C$ and
%$\gcd(\alpha_l,\alpha_m)=D$. 
 %It is clear that $BC\mid A$, $A_k=1$, $C\mid A_l$, $B\mid A_m$ and $BC\mid
 %A_i$ for $i\neq k,l,m$. 
% Let $A'=\frac{A}{BCD}$,  $A_k'=\frac{A_k}{CD}$,
%$A_l'=\frac{A_l}{CD}$, $A_m'=\frac{A_m}{DB}$ and $A_i'=\frac{A_i}{BCD}$
%if $i\neq k,l,m$. As before we get that $A'\geq A'_i$ for all $i$. So
%the equation becomes
%\begin{align}
%0&=2+(n-2)BCDA'-CDA'_k-CDA_l'-BDA_m'-\sum_{i\neq k,l,m}BCDA_i'\nonumber\\
%&\geq 2+(n-2)BCDA'-CDA'-CDA'-BDA'-\sum_{i\neq k,l,m}BCDA'\nonumber\\
%&=2+A'D(BC-2C-B).
%\end{align}
%Since $A'\geq1$ we get that $BC-B-2C<0$. We could also define
%$A''=\frac{A}{BCD}$,  $A_k''=\frac{A_k}{CD}$, 
%$A_l''=\frac{A_l}{BD}$, $A_m''=\frac{A_m}{DB}$ and $A_i''=\frac{A_i}{BCD}$
%if $i\neq k,l,m$. As before we get that $A''\geq A''_i$ for all
%$i$. This then give us the inequality $BC-2B-C<0$. Combining the two
%inequalities we get $2BC-3B-3C<0$. By symmetry we get get the
%inequalities $2BD-3B-3D<0$ and $2CD-3C-3D<0$, and the only solution to
%all three inequalities are if $(B,C,D)$ is one of the following
%$(b,1,1),(1,c,1),(1,1,d),(2,2,2),(3,2,2),(2,3,2)$ and $(2,2,3)$ for
%any $b,c,d>0$. In
%the case one of $B,C,D$ are $3$ and the other two are $2$, one gets
%that the genus is $2$ by direct calculation. Hence we have proved that
%if the genus of $\Sigma(\alpha_1,\dots,\alpha_n)$ is $0$, 
%one of the conditions holds. 
\end{proof}

Combining this result with an investigation of the inductive step
yields a necessary condition on the splice
diagram for the universal abelian cover to be a rational homology
sphere. We remember how we defined the notation $r_{v'}(v)$ in
\ref{notationaboutsee}.
%Let us introduce the following notation. Let $v$ be a vertex
%of $\Gamma$ and let $v'$ be a node of $\Gamma$, where $v\neq v'$. Then
%let  $r_{v'}(v)$ be the unique edge weight at an edge adjacent to $v'$
%which sees $v$. Likewise let $d_{v'}(v)$ be the unique ideal generator
%associated to $v'$, which sees $v$.    

\begin{cor}\label{ratinalhomunabcover}
Let $\Gamma$ be the the splice diagram of a manifold $M$, where the
universal abelian cover of $M$ is a rational homology sphere. Then all
edge weights are nonzero, and
there is a special node $v\in\Gamma$, with the following
properties. For all other nodes $v'\in\Gamma$, the weights other than
$r_{v'}(v)$ are pairwise coprime, and at most one of these
edge weights is not coprime with $r_{v'}(v)/d_{v'}(v)$. At $v$ all the
edge weights satisfy one of the conditions from Proposition
\ref{rationalhomologyspherebrieskorn}.  
\end{cor}

\begin{proof}
What we are going to show is that the condition on the splice diagram
given above is equivalent to the absence of cycles in the
decomposition graph (or a plumbing graph) of
the universal abelian cover $\widetilde{M}$, and all the
pieces of the decomposition having a base of genus $0$. The corollary then
follows by Proposition \ref{rationalhomologyspheres}. That the
decomposition graph must also have no cycles and bases of genus $0$
follows from the relation between plumbing graphs and decomposition
graph given in \cite{commensurability}.

We saw that, when we cut along an edge $e$ between nodes $v_0$ and
$v_1$ in the inductive
construction of $\widetilde{M}$ given in the proof of Theorem
6.3 in \cite{myarticle}, we took $d_0$ pieces above $v_0$ and glued to $d_1$
pieces above $v_1$, where $d_i$ is the ideal generator at $e$
associated to $v_i$. Each piece on the one side is glued exactly once
to each piece on the other side. 
Each of these pieces has a Seifert fibered piece
sitting above the corresponding $M_{v_i}$. If $d_0,d_1>1$
then a piece $v_{00}$  over $M_{v_0}$ is glued to a
piece $v_{10}$ sitting over $M_{v_1}$, then $v_{10}$ is glued
to a piece $v_{01}$ sitting over $M_{v_0}$, and $v_{01}$ is glued to
a piece $v_{11}$ sitting over $M_{v_1}$. Finally $v_{11}$ is glued
to $v_{00}$. We have now constructed a cycle in the decomposition
  graph of $\Delta(\widetilde{M})$ since each of the $v_{ij}$
  represent a vertex of $\Delta(\widetilde{M})$. If one of the
  $d_i$'s is 1, then we do not get cycles, since we will have only
  one piece above the appropriate end of $e$. 

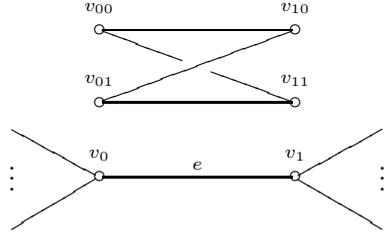
$$\splicediag{8}{30}{
& \overtag\Circ {v_{00}}{8pt}\lineto[rr]\lineto[dr] &&
\overtag\Circ{v_{10}}{8pt} &\\  
&&\text{\phantom{H} }\lineto[dr]&&\\ 
& \overtag\Circ {v_{01}}{8pt}\lineto[rr]\lineto[uurr] &&
\overtag\Circ{v_{11}}{8pt} &\\
&&&&\\
  \Vdots&\overtag\Circ {v_0} {8pt}\lineto[ul]
  \lineto[dl]
  \lineto[rr]^(.5){e}&& \overtag\Circ{v_1}{8pt}
  \lineto[ur]
  \lineto[dr]&\Vdots\\
  &&&&\hbox to 0 pt{~,\hss} }$$

So we now proved that a cycle in the decomposition graph for
$\widetilde{M}$ occurs if an edge $e$ in the splice diagram has ideal
generators $d_0$ and $d_1$ (associated to each end), such that both
$d_0$ and $d_1$ are not equal to one.

Let $M_0$ and $M_1$ be graph manifolds with universal abelian covers
$\widetilde{M}_0$ and $\widetilde{M}_1$, and assume that there are no
cycles in $\widetilde{M}_i$. Let $\widetilde{M}_{01}$ be the
universal abelian cover of $M_{01}$ which is $M_0$ glued to $M_1$
after removing a solid torus from each. Assume that $\widetilde{M}_{01}$
has cycles in its decomposition graph. $\widetilde{M}_{01}$ is a
number of $\widetilde{M}_{0}$ with $n_0$ solid tori removed glued to
$\widetilde{M}_{1}$ with $n_1$ solid tori removed, such that each of
the first type is glued to each of the second type. If one of the
$n_i$ is $1$, then $\widetilde{M}_{01}$ has no cycles, so
$n_0,n_1>1$. But $n_i=d_i$ so we are in the situation above.

So there are cycles in the decomposition graph of $\widetilde{M}$ if
and only if there is an edge which has both associated ideal
generators different from $1$.

We need to show
that the conditions we stated on $\Gamma$ are equivalent to the
statement that for each edge one of the ideal generators associated to
an end of it is $1$.

Suppose there were two nodes $v$ and $w$ of
$\Gamma$, such that 
the edge weights at $v$ that do not see $w$ are not pairwise coprime,
and the same with $v$ and $w$ exchanged. On any edge $e$ on the string between
$v$ and $w$, the ideal generator associated to either end of $e$
is then greater than $1$ by Proposition
\ref{idealgeneratorproperties}, so we a have  cycle in the decomposition
graph. This implies that there can be at most be one node $v$, such that
at all other nodes, edge weights that do not see $v$ are pairwise
coprime. On the other hand, if $\Gamma$ satisfies this, then it is not
hard to see that all ideal generators that do not see $v$ are $1$,
since all the edge weight they see at a node are pairwise coprime. 

We have so far shown that there are no cycles in the decomposition graph
of $\widetilde{M}$ if and only if there is a special node $v$ such
that at all other nodes the edge weights that do not see $v$ are
pairwise coprime. Next we have to see that our condition on $\Gamma$
also gives that all the pieces of the decomposition have genus $0$. 

Remember that when we do the induction in the proof of Theorem
6.3 in \cite{myarticle} and cut along an edge $e$
between $v_0$ and $v_1$, for be any node $v'$ in $\Gamma$ not equal
to $v_0$ or $v_1$, the weight $r_{v'}(v_i)$ gets replaced by
$r_{v'}(v_i)/d_{v'}(v_i)$, where $v_i$ ($i=0$ or $1$) is the node not
in the same piece as $v'$ after cutting.
When we cut $\Gamma$ along its edges, we do it in the following
way. Always choose an edge $e$ to an end node $w$, that is not the special node
$v$ to cut along. Then after the cutting we get two new pieces. The
first corresponds to the end node $w$ and has a one node splice
diagram with as many edges as $w$ had in $\Gamma$, and the edges have
the same weights, except $r_w(v)$ is divided by $d_w(v)$. The
splice diagram of the other piece $\Gamma_e$ looks like $\Gamma$ with
the node $w$ replaced by a leaf, and no edge weight is changed since
all the $d_{v'}(w)=1$ for any node $v'$. We then find an end node
of $\Gamma_e$ which is not $v$ to cut along, and repeat
until we have cut along all the edges between nodes.    
 
We have now cut $\Gamma$ into a collection of one-node splice
diagrams. Each of these will contribute at least one Seifert fibered piece to
$\widetilde{M}$, (the same one-node splice diagram may of course
contribute with the same Seifert fibered piece of $\widetilde{M}$  more
than once). 
We distinguish the piece corresponding to our special node $v$. The
pieces not corresponding to $v$ have splice diagrams with the same
weights as in $\Gamma$, except $r_w(v)$ is replaced by
$r_w(v)/d_w(v)$. Our assumptions on the
$\Gamma$ then imply that all the weights are pairwise
coprime, except possibly two weights who are pairwise coprime with the rest,
but might have a common divisor. Since the Seifert fibered pieces
corresponding to each of the nodes are the Brieskorn complete
intersections defined by the edge weights, so condition one or
two of Proposition \ref{rationalhomologyspherebrieskorn} holds. Then the
Seifert fibered pieces of the decomposition of $\widetilde{M}$
corresponding to these nodes are rational homology spheres.

%If we did not have the assumption
%n the edge weights, i.e.\ if the edge weight that saw $v$ divided by
%he ideal generator had a common divisor with more than one of the
%ther say with weights $n,n'$, then we would not be in any of the
%ases of \ref{rationalhomologyspherebrieskorn} since $\gcd(n,n')=1$ so
%e could not be in case three. 

The special piece of the decomposition of $\widetilde{M}$
(corresponding to $v$, there will in fact only be one), has genus $0$,
since the assumption on $\Gamma$ are equivalent to the Brieskorn
complete intersection being genus $0$, by proposition
\ref{rationalhomologyspherebrieskorn}.  

Hence the assumptions on $\Gamma$ are equivalent to the decomposition
graph of $\widetilde{M}$ having no cycles, and all the pieces of the
decomposition having a base of genus $0$.
\end{proof}

The converse to the corollary does not immediately follow, since
having no cycles and having genus $0$ pieces are only two of the three
conditions for a graph 
manifold to be a rational homology sphere. The last one (as we saw in
proposition \ref{rationalhomologyspheres}) is that the intersection
matrix $I$ must have non zero determinant. Proving that $\det(I)\neq
0$, reduces to a simpler problem since Neumann showed in
\cite{commensurability} that, 
by doing row and column additions, $I$ becomes the direct sum of the
decomposition matrix and a number of $1\times 1$ matrices with non
zero entries. Hence it is enough to show that the determinant of the
decomposition matrix is non zero. To do this we need the following
lemma describing the fiber intersection numbers in the universal
abelian cover from the splice diagram.

\begin{prop}\label{fiberintersection}
Let $v_0$ and $v_1$ be two nodes of $\Gamma(M)$ connected by an edge
$e$, decorated 
as below. If there are no edge weights of $0$ adjacent to any of the
$v_i$'s, then the fiber intersection number $\tilde{p}$ in any torus
in the universal abelian cover sitting above $T_e$ is
\begin{align*}
\tilde{p}&=\frac{\num{D(e)}}{\overline{d}_0\overline{d}_1b_0b_1},
\end{align*}
where $\overline{d}_i$ are ideal generator corresponding to $r_i$ and
\begin{align*}
b_i&=\frac{r_i/\overline{d}_i\lcm(n_{i1}/\overline{d}_{i1},
  \dots,n_{ik_i}/\overline{d}_{ik_i})}{\lcm(n_{i1}/\overline{d}_{ii},
  \dots,n_{ik_i}/\overline{d}_{ik_i},r_i/\overline{d}_i)},
\end{align*}  
again the $\overline{d}_{ij}$ are the ideal generators.
\end{prop} 
$$\splicediag{8}{30}{
  &&&&\\
  \Vdots&\overtag\Circ {v_0} {8pt}\lineto[ul]_(.5){n_{01}}
  \lineto[dl]^(.5){n_{0k_0}}
  \lineto[rr]^(.25){r_0}^(.75){r_1}&& \overtag\Circ{v_1}{8pt}
  \lineto[ur]^(.5){n_{11}}
  \lineto[dr]_(.5){n_{1k_1}}&\Vdots\\
  &&&&\hbox to 0 pt{~.\hss} }$$ 
\begin{proof}
Let $f_0$ and $f_1$ be fibers from each of the sides in $T_e$, and let
$p$ be the fiber intersection number in $T_e$ i.e.\ $p=f_0\cdot f_1$.
It follows from the Edge Determinant Equation
\ref{edgedeterminantequation} that $p=\num{D(e)}/\num{H_1(M)}$. Let
$\morf{\pi}{\widetilde{M}}{M}$ be the universal abelian cover, and let
$\widetilde{T}\widetilde{M}$ be a connected component of
$\pi\inv(T_e)$. Then the intersection number of the preimage of $\pi$
restricted to $\widetilde{T}$ is the intersection number
  before multiplied by the degree of the map restricted map i.e.\
  $\pi\vert_{\widetilde{T}}\inv(f_0)\cdot\pi\vert_{\widetilde{T}}\inv(f_1) =
  p\deg(\pi\vert_{\widetilde{T}})$. Since the $\pi$ is the
  universal abelian cover its degree is $\num{H_1(M)}$ and hence
  $\deg(\pi\vert_{\widetilde{T}})=\num{H_1(M)}/t$ where $t$ is the number
  of components of $\pi\inv(T_e)$, and using the
  edge determinant equation we get that
  $\pi\vert_{\widetilde{T}}\inv(f_0)\cdot\pi\vert_{\widetilde{T}}\inv(f_1) 
  =\num{D(e)}/t$. Notice that $t=\overline{d}_0\overline{d}_1$, this
  follows from the proof of Theorem 6.3 
in \cite{myarticle} and was also used in the proof Corollary
\ref{ratinalhomunabcover}. 

Now $\pi\vert_{\widetilde{T}}\inv(f_i)$ consist of a collection of
fibers $\tilde{f}_i$, and hence using the biliniarity of the
intersection product we get that
$\num{D(e)}/t=(\#\pi\vert_{\widetilde{T}}\inv(f_0))
(\#\pi\vert_{\widetilde{T}}\inv(f_1))\tilde{f}_0\cdot\tilde{f}_1$. Since
$\tilde{f}_0\cdot\tilde{f}_1=\tilde{p}$ we just need to calculate
$\#\pi\vert_{\widetilde{T}}\inv(f_i)$. 

Let $\widetilde{M}_i\subset \widetilde{M}-\pi\inv(T_e)$ be a connected
component sitting above
$v_i$. Hence need need to determine how many
copies of $\tilde{f}_i$ sits in each of the boundaries of
$\widetilde{M}_i$. Remember that the Seifert fibered piece of
$\widetilde{M}_i$ sitting above $v_i$ is the Brieskorn complete
intersection $\Sigma=\Sigma(n_{i1}/\overline{d}_{i1},
\dots,n_{ik_i}\overline{d}_{ik_i},r_i/\overline{d}_i)$ where
a tubular neighborhood around all the singular fibers $o_i$ corresponding
to $r_i/\overline{d}_i$ are removed. $\tilde{f}_i$ is a non singular
fiber of $\Sigma$, and hence
by the proof of Theorem 8.2 in \cite{brandies}
$\morf{\pi\vert_{\tilde{f}_i}}{\tilde{f}_i}{f_i}$ has degree
$a_i\num{e}$, where $e$ is
the rational euler number of $\Sigma$ and
$a_i=\lcm(n_{i1}/\overline{d}_{i1},\dots,n_{ik_i}\overline{d}_{ik_i}
,r_i/\overline{d}_i)$. Since
$\pi$ restricted to the Seifert fibered piece above $v_i$ is the the
same as the restriction of the universal abelian of
$\Sigma$ its degree is
$\num{e}r_i/\overline{d}_i\prod_jn_{ij}/\overline{d}_{ij}$, and
hence there are
$\frac{r_i/\overline{d}_i\prod_jn_{ij}/\overline{d}_{ij}}{a_i}$ copies
of $\tilde{f}_i$ in $\Sigma$. These $\tilde{f}_i$ all sit in the
boundaries when we remove the tubular neighborhoods of the fibers
sitting above $o_i$, and by symmetry each of the boundary components of
of $\widetilde{M}_i$ has an equal number of copies. Since the number
of fibers above $o_i$ is
$(\prod_jn_{ij}/\overline{d}_{ij})/\lcm(n_{ij}/\overline{d}_{ij})$,
and we get that
\begin{align*}
\#\pi\vert_{\widetilde{T}}\inv(f_i)
&=\frac{r_i/\overline{d}_i\lcm(n_{i1}/\overline{d}_{i1},  
  \dots,n_{ik_i}/\overline{d}_{ik_i})}{\lcm(n_{i1}/\overline{d}_{ii},
  \dots,n_{ik_i}/\overline{d}_{ik_i},r_i/\overline{d}_i)},
\end{align*}  
and the formula follows.
\end{proof}

\begin{prop}\label{intersectionform}
Let $M$ be a graph orbifold whose splice diagram $\Gamma(M)$ satisfies the
conditions of Corollary \ref{ratinalhomunabcover}, then the
intersection form of the universal abelian cover $\widetilde{M}$ of
$M$ is non degenerate.
\end{prop}

\begin{proof}
Remember from the earlier discussion that we only need to show that
the decomposition matrix is non degenerate. The decomposition matrix
has as diagonal entries the rational euler number of the pieces of the
JSJ-decomposition, and on off diagonal entries is $1/p$ where $p$ is
the fiber intersection number if the corresponding pieces are
connected by an edge. 

It proof is going to be by induction by the number of nodes in
$\Gamma(M)$. If $\Gamma(M)$ only has one node, then $\widetilde{M}$ is
a Brieskorn complete intersection since we have no weights of value
$0$, and hence its intersection matrix is negative definite and
therefore non degenerate.

So let $\Gamma(M)$ have $n$ nodes. Let $v$ be an end node other that
the special node, that means that $v$ is only connected to one other
node, call this node $w$ and let $v'$ be the special node, $v'$ can be
equal to $w$. Assume we have named the weights in the following way
$$\splicediag{8}{30}{
&&&&&\\
&&&&&\Vdots\\
\Circ&&&&&\\
  \Vdots&\overtag\Circ {v} {8pt}\lineto[ul]_(.5){n_1}
  \lineto[dl]^(.5){n_k}
  \lineto[rr]^(.25){r}^(.75){s}&& \overtag\Circ{w}{8pt}
  \lineto[uuurr]^(.5){m_1}\lineto[urr]^(.5){m_l}
\lineto[dddrr]_(.5){r_1}\lineto[drr]_(.5){r_{l'}}
  \lineto[rr]^(.5){m}&&\\
\Circ&&&&&\Circ\\ 
&&&&&\Vdots\\ 
&&&&&\Circ
\hbox to 0 pt{~,\hss} }$$
 where the edges weighted with $m_i$ and $m$ leads to other nodes, and
 if $w\neq v'$ then the edge with $m$ one sees $v'$. Let
 $N=\prod_in_i$, $M=\prod_im_i$ and $R=\prod_ir_i$.
The conditions on
 $\Gamma(M)$ implies that all the ideal generators
 except maybe $\overline{d}_r$ and $\overline{d}_m$ are $1$, and that
 $\gcd(n_i,n_j)=1$ and
 $\gcd(r_i,r_j)=\gcd(m_i,m_j)=\gcd(r_i,m_j)=\gcd(s,r_j)=\gcd(s,m_j)=1$.
 $\gcd(r/\overline{d}_r,n_i)=1$  
   except maybe for one of the $n_i$ call this $n_{i_0}$, and assume
   $\gcd(r/\overline{d}_r,n_{i_0})=b$, likewise
     $\gcd(m/\overline{d}_m,m_i)=\gcd(m/\overline{d}_m,r_i)
         =\gcd(m/\overline{d}_m,s)= 1$ except maybe for one of the
           $m_i$'s $r_i$'s or $s$. Let the value of the $\gcd$ not
           being $1$ be $c$, and notice that if
           $\gcd(m/\overline{d}_m,s)= c$ then
             $\overline{d}_r=\overline{d}_m$ els
             $\overline{d}_r=c\overline{d}_m$.  

Above $v$ in $\widetilde{M}$ sits $\overline{d}_r$ identical Seifert fibered
pieces $\tilde{v}$, and above $w$ sits $\overline{d}_m$ identical
Seifert fibered 
pieces $\tilde{w}$. Each of the $\tilde{w}$ is connected two
$\overline{d}_r/\overline{d}_m$ of the $\tilde{v}$'s, by an edge. This
implies that in the decomposition matrix $A$ has $\overline{d}_m$
blocks looking like
\begin{align*}
&\begin{pmatrix}
&&\vdots&\vdots&&\vdots& & &\\
&&0&0&\dots&0& & &\\
\dots&0&e_{\tilde{v}}&0&\dots&0&\frac{1}{\tilde{p}}&0&\dots\\
\dots&0&0&e_{\tilde{v}} & &0&\frac{1}{\tilde{p}}&0&\dots \\
&\vdots&\vdots& &\ddots&\vdots&\vdots&\vdots& \\
\dots&0&0&0&\dots&e_{\tilde{v}}&\frac{1}{\tilde{p}}&0&\dots\\
&&\frac{1}{\tilde{p}}&\frac{1}{\tilde{p}} &\dots &\frac{1}{\tilde{p}}
&e_{\tilde{w}}&\frac{1}{\tilde{p}'}& \\ 
&&0&0&\dots&0&\frac{1}{\tilde{p}'}&\ddots&\\
&&\vdots&\vdots&&\vdots&&&\\
\end{pmatrix}\\
&\hspace{1.6cm}\underbrace{\hspace{2.7cm}}_{\overline{d}_r/\overline{d}_m}
\end{align*}
where $e_{\tilde{v}}$ and  $e_{\tilde{w}}$ are the rational euler
numbers of $\tilde{v}$ and $\tilde{w}$, and $\tilde{p}$ is the fiber
intersection number in the edges. We can calculate $\tilde{p}$ using
\ref{fiberintersection} and gets that
$\tilde{p}=\num{D(e)}/bc\overline{d}_m$, the reason that it is
$\overline{d}_m$ and not $\overline{d}_r$ in the formula, is that
using $\overline{d}_r$ gives two different formulas depending on
whether $\gcd(m/\overline{d}_m,s)=c$ or not, but using the relation
ship between $\overline{d}_m$ and $\overline{d}_r$ to replace
$\overline{d}_m$ with $\overline{d}_r$ makes the formulas the same. To
calculate $e_{\tilde{v}}$ and  $e_{\tilde{w}}$ we use the formula
given in the end of the proof of 6.3 in \cite{myarticle}, which gives
that
$e_{\tilde{v}}=\tfrac{\lambda_v^2}{\overline{d}_r}e_v/\num{H_1(M)}$ and
$e_{\tilde{w}}=\tfrac{\lambda_w^2}{\overline{d}_s}e_w/\num{H_1(M)}$,
where
$\lambda_v=Nr/\lcm(n_1,\dots,n_k,r/\overline{d}_s)=b\overline{d}_r$
and $\lambda_m=MRsm/\lcm(m_1,\dots,m_l,r_1,
\dots,r_{l'},s,m/\overline{d}_m)=c\overline{d}_m$. We find
$e_v/\num{H_1(M)}$ and $e_w/\num{H_1(M)}$ by using the formula of
Proposition 3.4 in \cite{myarticle}. This gives that 
\begin{align*}
e_v/\num{H_1(M)}&=-\frac{\epsilon_vs}{ND(e)}, \ 
&e_w/\num{H_1(M)}=-\frac{\epsilon_wm'}{MD(m)}-\frac{\epsilon_vN}{sD(e)}-E,
\end{align*}  
where $\epsilon_v$ and $\epsilon_w$ are the signs at the nodes, $D(m)$
is the edge determinant of the edge with $m$ on it $m'$ is the weight
on the other end of that edge, and $E$ is a sum of contributions from
the nodes seen be the $r_i$'s which dose not include any factors
coming from $v$. This give the following values for $e_{\tilde{v}}$
and $e_{\tilde{w}}$
\begin{align*}
e_{\tilde{v}}&=-\frac{\epsilon_vsb^2\overline{d}_r}{ND(e)}, \ 
&e_{\tilde{w}}=-c^2\overline{d}_m
(\frac{\epsilon_wm'}{MD(m)}+\frac{\epsilon_vN}{sD(e)}+E).
\end{align*}  

We can clear all the $1/\tilde{p}$ in the row and column containing
$e_{\tilde{w}}$ by using the rows and columns with the $e_{\tilde{v}}$
on the diagonal whit out changing anyting other that the entry with
$e_{\tilde{w}}$, hence our blocks will now look like
\begin{align*}
\begin{pmatrix}
&&\vdots&\vdots&&\vdots& & &\\
&&0&0&\dots&0& & &\\
\dots&0&e_{\tilde{v}}&0&\dots&0&0&0&\dots\\
\dots&0&0&e_{\tilde{v}} & &0&0                  &0&\dots \\
\dots&0&\vdots& &\ddots&\vdots&\vdots&\vdots& \\
\dots&0&0&0&\dots&e_{\tilde{v}}& 0           &0&\dots\\
&&0&0&\dots &0
&e_{\tilde{w}}-\tfrac{\overline{d}_r}{\overline{d}_m}
\frac{1}{\tilde{p}^2e_{\tilde{v}}}  
&\frac{1}{\tilde{p}'}& \\ 
&&0&0&\dots&0&\frac{1}{\tilde{p}'}&\ddots&\\
&&\vdots&\vdots&&\vdots&&&\\
\end{pmatrix}.
\end{align*}
This implies that
$A$ is row and column equivalent to
$A'\oplus\Big(\bigoplus_{i=1}^{\overline{d}_r}(e_{\tilde{v}})\Big)$, 
where $A'$ is equal to $A$, except the block has been replaced be a
single entry of $e_{\tilde{w}}-\tfrac{\overline{d}_r}{\overline{d}_m}
\frac{1}{\tilde{p}^2e_{\tilde{v}}}$. Since the $1\times 1$ matrix
$(e_{\tilde{v}})$ has a non zero entry, $A$ is non degenerate if and
only if $A'$ is non degenerate. So lets calculate the difference
between $A$ and $A'$
\begin{align*}
e_{\tilde{w}}-\tfrac{\overline{d}_r}{\overline{d}_m}
\frac{1}{\tilde{p}^2e_{\tilde{v}}} &= -c^2\overline{d}_m
(\frac{\epsilon_wm'}{MD(m)}+\frac{\epsilon_vN}{sD(e)}+E)
+\tfrac{\overline{d}_r}{\overline{d}_m}
\frac{c^2\overline{d}_m^2N\epsilon_v}{sD(e)\overline{d}_r} \\ &=
-c^2\overline{d}_m
(\frac{\epsilon_wm'}{MD(m)}+\frac{\epsilon_vN}{sD(e)}+E)
+c^2\overline{d}_m\frac{\epsilon_vN}{sD(e)} 
\\ &== -c^2\overline{d}_m
(\frac{\epsilon_wm'}{MD(m)}+E).
\end{align*}      
But this is excatly the rational euler number of the seifert fibered
pieces in the universal abelian cover of the manifold $M'$ with splice
diagram $\Gamma(M')$ sitting above the node $w$, where 
$$\splicediag{8}{30}{
&&&&\\
&&&&\vdots\\
&&&&\\
  \Gamma(M')=&\Circ\lineto[r]^(.75){s}& \overtag\Circ{w}{8pt}
  \lineto[uuurr]^(.5){m_1}\lineto[urr]^(.5){m_l}
\lineto[dddrr]_(.5){r_1}\lineto[drr]_(.5){r_{l'}}
  \lineto[rr]^(.5){m}&&\\
&&&&\Circ\\ 
&&&&\Vdots\\ 
&&&&\Circ
\hbox to 0 pt{~,\hss} }$$ 
the rest of $\Gamma(M')$ is identical to $\Gamma(M)$. It is not hard
to see that $\Gamma(M')$ satisfy the conditions of  Corollary
\ref{ratinalhomunabcover}. Since all the ideal generators in
$\Gamma(M)$ that sees $v$ are $1$, all entries in the decomposition
matrix of the universal ablian cover of $M'$ are the same as in the
universal abeliancover of $M$ except the one above $v$ and $w$, and
hence $A'$ is the decomposition matrix of the universal abelian cover
of $M'$. This implies that $A'$ is non degenerate by the induction
hypothesis, and 
hence $A$ is non degenerate and the intersection form of the universal
abelian cover of $M$ is non degenerate.   
\end{proof}

We can now summerize the above proposition and Corollary
\ref{ratinalhomunabcover} to the following result.

\begin{thm}
Let $\Gamma$ be the the splice diagram of a manifold $M$, then the
universal abelian cover of $M$ is a rational homology sphere if and
only if all edge weights are nonzero, and
there is a special node $v\in\Gamma$, with the following
properties. For all other nodes $v'\in\Gamma$, the weights other than
$r_{v'}(v)$ are pairwise coprime, and at most one of these
edge weights is not coprime with $r_{v'}(v)/d_{v'}(v)$. At $v$ all the
edge weights satisfy one of the conditions from Proposition
\ref{rationalhomologyspherebrieskorn}.  
\end{thm}
%\begin{cor}
%If $M$ is a rational homology sphere singularity link, then the
%universal abelian cover of $M$ is a rational homology sphere if and
%only if the splice diagram $\Gamma$ satisfies the conditions of
%corollary \ref{ratinalhomunabcover}.
%\end{cor}      

%\begin{proof}
%We only have to prove that
%the intersection matrix $I$, satisfies that $\det(I)\neq 0$. Note that the
%universal abelian cover of a singularity link is also a singularity
%link, and hence by Theorem \ref{grauertsthm}, $I$ is negative definite.
%Therefore $\det(I)\neq 0$.
%\end{proof}

\newpage

\bibliography{rationalabeliancover}

\end{document}